\documentclass[a4paper,11pt]{amsart}
\usepackage[pdftex]{color,graphicx}
\usepackage{amssymb}
\usepackage{amsfonts}
\usepackage{esint}
\usepackage{amsmath}
\usepackage{amsrefs}
\usepackage{epsfig}
\usepackage{amsthm}
\usepackage{color}
\usepackage[]{epsfig}
\usepackage[]{pstricks}
\usepackage[utf8]{inputenc}
\usepackage{setspace}

\newtheorem{theorem}{Theorem}[section]
\newtheorem{proposition}{Proposition}[section]
\newtheorem{lemma}{Lemma}[section]
\newtheorem{definition}{Definition}[section]
\newtheorem{example}{Example}[section]
\newtheorem{corollary}{Corollary}[section]
\newtheorem{remark}{Remark}[section]

\newcommand{\s}{\mathbb{S}}
\newcommand{\R}{\mathbb{R}}

\newcommand{\Ric}{Ric}
\newcommand{\tr}{tr}
\newcommand{\Vol}{Vol}
\newcommand{\di}{div}

\newcommand{\Ind}{Ind}
\newcommand{\dist}{{\rm dist}\,}
\newcommand{\Span}{{\rm Span}\,}

\numberwithin{equation}{section}
\title[$f$-Stability Index of the Constant Weighted Mean Curvature Hypersurfaces]{The $f$-Stability Index of the Constant Weighted Mean Curvature Hypersurfaces in Gradient Ricci Solitons}
\author{Hilário Alencar}
\author{Adina Rocha}

\date{December 20, 2016}

\begin{document}
	
\footnotetext{The first author was partially supported by CNPq of Brazil. The second author was partially supported by CAPES of Brazil.}


\begin{abstract}
In this paper, we prove that a noncompact complete hypersurface with finite weighted volume, weighted mean curvature vector bounded in norm, and isometrically immersed in a complete weighted manifold is proper. In addition, we obtain an estimate for $f$-stability index of a constant weighted mean curvature hypersurface with finite weighted volume and isometrically immersed in a shrinking gradient Ricci soliton that admits at least one parallel field globally defined. For such hypersurface, we still give a necessary condition for equality to be achieved in the estimate obtained.\\
\end{abstract}


\maketitle

\section{Introduction}
\label{intro}

Consider a $(n+1)$-dimensional Riemannian manifold $(\overline M^{n+1},g)$ endowed with a weighted measure of the form $e^{-f}d\mu$, where $f$ is a smooth function on $M$ and $d\mu$ is the volume element induced by the Riemannian metric $g$. A {\it weighted manifold} is a triple $${\overline M}_f^{n+1}=( \overline M^{n+1},g,e^{-f}d\mu).$$
 A natural extension of the Ricci tensor to this new context is the Bakry-\'{E}mery Ricci tensor, see \cite{BakryEmery1985}, given by
 $$\overline{\Ric}_f=\overline{\Ric}+\overline{\nabla}^{2} f,$$
 where \ $\overline{\nabla}^{2} f$ \ is the Hessian of $f$ on $\overline M^{n+1}$. It is known that a complete weighted manifold satisfying $\overline{\Ric}_f\geq\displaystyle{k}g$ for some constant $k>0$ is not necessarily compact. One of the examples is the Gaussian shrinking soliton $\left(\R^{n+1},g_{can},e^{-\frac{|x|^2}{4}}d\mu\right)$ with the canonical metric $g_{can}$ and $\overline{\Ric}_f=\displaystyle{\frac{1}{2}g_{can}}.$ 
 
 The gradient Ricci solitons are natural generalizations of the Einstein metrics and was introduced by Hamilton in \cite{Hamilton1988}. Indeed, a complete Riemannian metric $g$ on a smooth manifold $\overline M^{n+1}$ is a {\it shrinking gradient Ricci soliton} if there exists a potential function $f$, and a real constant $k>0$ such that
 the Ricci tensor $\overline{\Ric}$ of the metric $g$ satisfies the equation
 \begin{equation}
 \overline{\Ric}+\overline{\nabla}^2f=kg,
 \end{equation}
 where $\overline{\nabla}^2f$ denotes the Hessian of $f$. In this context, the gradient Ricci solitons are complete weighted manifolds with $\overline{\Ric}_f=kg$ for some constant real $k$. 
 
 Observe that when the potential function is a constant, the gradient Ricci solitons are simply Einstein metrics. It is still important to mention that gradient Ricci solitons plays an important role in Hamilton's Ricci flow and correspond to self-similar solutions, and often arise as
 Type I singularity models. For more details see \cite{Hamilton1982}.
 
 Let $x:M^n\rightarrow\overline M^{n+1}_f$ be an isometric immersion of a Riemannian orientable manifold $M^n$ into weighted manifold $\overline M^{n+1}_f.$ The function $f:\overline M \rightarrow\R$, restricted to $M$, induces a weighted measure $e^{-f}d\sigma$ on $M$. Thus, we have an induced weighted manifold $M_f^n=(M,\langle\,,\,\rangle,e^{-f}d\sigma)$.
 
The {\it second fundamental form} $A$ of $x$ is defined by
 $$A(X,Y)=(\overline\nabla_XY)^{\perp}, \ \ \ \ \ \ X,Y\in T_pM, \ \ p\in M,$$
 where $\perp$ symbolizes the projection above the normal bundle of $M$. The {\it weighted mean curvature vector} of $M$ is defined by
 $${\bf H}_f={\bf H}+(\overline\nabla f)^{\perp},$$
 and its the {\it weighted mean curvature} $H_f$ is given by $${\bf H}_f=-H_f\eta,$$ 
 where ${\bf H}=\tr A$ and $\eta$ is unit outside normal vector field. The hypersurface $M$ is called {\it $f$-minimal} when its weighted mean curvature vector ${\bf H}_f$ vanishes identically, and when there exists real constant $C$ such that ${\bf H}_f=-C\eta$, we say the hypersurface $M$ has {\it constant weighted mean curvature}. 
 
 The {\it weighted volume} of a measurable set $\Omega\subset M$ is given by $$\Vol_f(\Omega)=\int_{\Omega}e^{-f}d\sigma.$$

Let $B^{\overline M}_r$ be the geodesic ball of $\overline M$ with center in a fixed point $o\in \overline M$  and radius $r>0$. It is said that the weighted volume of $M$ has {\it polynomial growth} if there exists positive numbers $\alpha$ and $C$ such that 
\begin{equation}\label{equation10}
\Vol_f(B^{\overline M}_r\cap M)\leq Cr^{\alpha}
\end{equation} 
for any $r\geq1.$ When $\alpha=n$ in (\ref{equation10}), $M$ is said to have Euclidean volume growth.

We can consider either $f$-minimal or constant weighted mean curvature hypersurfaces in gradient Ricci solitons. In particular, a self-shrinker to the mean curvature flow is a $f$-minimal hypersurface of the shrinking Gaussian soliton. In \cite{ChengZhou2013}, Cheng and Zhou showed that for $f$-minimal hypersurfaces in the shrinking Gaussian soliton $\left(\R^{n+1},g_{can},e^{-\frac{|x|^2}{4}}dx\right)$, the properness of its immersion, its polynomial volume growth, and its finite weighted volume are equivalent to each other. Those equivalences are still being valid for $f$-minimal hypersurfaces immersed in a complete shrinking gradient Ricci soliton $\overline M_f$ satisfying $\overline{\Ric}_f=\dfrac{1}{2}g$, where $g$ is Riemannian metric and $f$ is a convex function (see  \cite{ChengMejiaZhou2015}). In this direction, the following result was obtained: 

\begin{proposition}\label{ip4}
	Let $M^n$ be a noncompact complete hypersurface isometrically immersed in a complete weighted manifold $\overline M^m_f$, with weighted mean curvature vector bounded in norm. If $M^n$ has finite weighted volume, then $M^n$ is proper.
\end{proposition}

\begin{proposition} \label{ip1} 
	Let $f\in C^{\infty}(\overline M)$ be a convex function, $\overline M^m_f$ a gradient Ricci soliton with $\overline{\Ric}_f=\dfrac{1}{2}g$, and  $x:M^n\rightarrow\overline M^m_f$ a noncompact complete immersion with weighted mean curvature vector satisfying
	$$\sup_{x\in M}\langle {\bf H}_f,\overline{\nabla}f\rangle<\infty.$$
	If $M^n$ is proper, then it has finite weighted volume and Euclidean volume growth.
\end{proposition} 

Now it is known that the {\it weighted Laplacian operator} $\Delta_f$, defined by
$$\Delta_fu:=\Delta u-\langle\nabla f,\nabla u\rangle,$$
is associated to $e^{-f}d\sigma$ as well as $\Delta$ is associated to $d\sigma.$ Moreover, $\Delta_f$ is a self-adjoint operator on the  $L^{2}_f$ space of square integrable functions on $M$ with respect to the measure $e^{-f}d\sigma,$ and therefore, the $L^{2}_f$ spectrum of $\Delta_f$ on $M$, denoted by $\sigma(-\Delta_f)$, is a subset of $[0,+\infty).$
 
Next, let $F:(-\varepsilon,\varepsilon)\times M\rightarrow \overline M_f$, $F_f(p)=F(t,p)$ for all $t\in(-\varepsilon,\varepsilon)$ and $p\in M$ be a variation of the immersion $x$ associated with the normal vector field $u\eta$, where $u\in C^{\infty}_c(M).$  The corresponding variation of the {\it functional weighted area} $\mathcal A_f(t)=\Vol_f(F_t(M))$ satisfies 
\begin{equation} \label{ie1}
 \mathcal A_f'(0)=\int_M H_f ue^{-f}d\sigma,	
  \end{equation}
 where $H_f$ is such that ${\bf H}_f=-H_f\eta$. The expression (\ref{ie1}) is known as {\it first variation formula}. 
  
 The $f$-minimal hypersurfaces are critical points of the functional weighted area. Yet, the hypersurfaces with constant weighted mean curvature can be viewed as critical points of the functional weighted area restricted to variations which preserve the {\it enclose weighted volume,} i.e., for functions $u\in C^{\infty}_c(M)$ which satisfy the additional condition $$\int_Mue^{-f}d\sigma=0.$$ For such critical points, the {\it second variation} of the functional weighted area is given by
 $$\mathcal A_f''(0)=-\int_M\left(u\Delta_fu+\left(|A|^2+\overline{\Ric_f}(\eta,\eta)\right)u^2\right)d\sigma,$$
 where $\overline{\Ric}_f$ is the Bakry-\' Emery Ricci curvature and $A$ is the second fundamental form. For more details, see \cite{Rocha2016}.
 
\begin{remark}
When $f$ is a constant function, the first and second variation formula were given by Barbosa and do Carmo \cite{BarbosadoCarmo1984} and Barbosa, do Carmo and Eschenburg \cite{BarbosadoCarmoEschenburg1988}.    
\end{remark}
 
The operator 
 $$L_f=\Delta_f+|A|^2+\overline{\Ric_f}(\eta,\eta)$$
 is called the {\it $f$-stability operator} of the immersion $x$. In the $f$-minimal case, the $f$-stability operator is viewed as acting on $\mathcal F=C_c^{\infty}(M)$; in the case of the hypersurfaces with constant weighted mean curvature, the $f$-stability operator is viewed as acting on $$\mathcal F=C_c^{\infty}(M)\cap \left\{u\in C^{\infty}_c(M); \ \int_Mue^{-f}d\sigma=0\right\}.$$
 Associated with $L_f$ is the quadratic form
 $$I_f(u,u)=-\int_MuL_fue^{-f}d\sigma.$$
 For each compact domain $\Omega\subset M$, define the index, $\Ind_f\Omega$, of $L_f$ in $\Omega$ as the maximal dimension of a subspace of $\mathcal F$ where $I_f$ is a negative definite. The {\it index}, $\Ind_fM$, of $L_f$ in $M$ (or simply, the index of $M$) is then defined by
 $$\Ind_fM=\sup_{\Omega\subset M}\Ind_f\Omega,$$
 where the supreme is taken over all compact domains $\Omega\subset M.$  For more details, see \cite{Fischer-Colbrie1985} and \cite{ChengZhou2015}.
 
 Let $M\subset\R^{n+1}$ be a proper, non-planar, two-sided hypersurface satisfying $\Vol_f(M)<\infty$, $H=\dfrac{1}{2}\langle x,\eta\rangle+C$ and $\Ind_f(M)\leq n,$ where $H$ is the mean curvature, $x$ is the position vector of $\R^{n+1}$, $\eta$ is the unit normal field of the hypersurface, $\Ind_f(M)$ is the $f$-stability index and $C$ is a real constant. McGonagle and Ross (\cite{McGR2015}, Theorem 5.6) showed that exists a natural number $i$ such that  $n+1-\Ind_f(M)\leq i\leq n$ e $\Sigma=\Sigma_0\times\R^i.$
In addition, they obtained $\Ind_f(M)\geq2$.

It is important to mention that the properness hypothesis can be removed of the Theorem 5.6 of \cite{McGR2015}. In fact, by Proposition \ref{ip4}, the finite weighted volume implies in the properness of its immersion.

Next, we obtain an estimate for $f$-stability index of a constant weighted mean curvature hypersurface with finite weighted volume and isometrically immersed in a gradient Ricci soliton that admit at least one parallel field globally defined. If fact,

\begin{theorem}\label{it2}
	Let $\overline M_f$ be a shrinking gradient Ricci soliton with $\overline{\Ric}_f=kg$. Let $M^n$ be a constant weighted mean curvature hypersurface with finite weighted volume and isometrically immersed in $\overline M_f$. Denote by $\mathcal P_{\overline M_f}$ the set of parallel fields globally defined on $\overline M_f^{n+1}$ and $\eta$ the unit normal field to $M$.  
	\begin{enumerate}
		\item[(i)] If the unit function $1\not\in \{\langle X,\eta\rangle: \ {X\in\mathcal P_{\overline M_f}}\}$, 
		\begin{equation}\label{iequation8}
		\Ind_f(M)\geq\dim \mathcal P_{\overline M_f}-\dim\{X\in P_{\overline M_f}: \ \langle X,\eta\rangle\equiv0\}. 
		\end{equation} 
		\item[(ii)] If the unit function $1\in \{\langle X,\eta\rangle: \ {X\in\mathcal P_{\overline M_f}}\}$, $M$ is totally geodesic.
	\end{enumerate}
\end{theorem}

As a consequence of Theorem \ref{it2}, we have 

\begin{corollary}\label{ic1} Let $M^n$ be a constant weighted mean curvature hypersurface with finite weighted volume and isometrically immersed in a shrinking gradient Ricci soliton $\overline M_f$. Denote by $\mathcal P_{\overline M_f}$ the set of parallel fields globally defined on $\overline M_f^{n+1}$ and let $\eta$ the unit normal field to $M$. If the unit function $1\not\in \{\langle X,\eta\rangle: \ {X\in\mathcal P_{\overline M_f}}\}$ and there exist a parallel field $X_0$ such that $\langle X_0,\eta\rangle\not\equiv0$, then
	$$\Ind_f(M)\geq1.$$
	Moreover, 
	$$\dim\{X\in \mathcal P_{\overline M_f}: \ \langle X,\eta\rangle\equiv0\}=\dim\mathcal P_{\overline M_f}-1$$
	whenever $\Ind_f(M)=1.$		
\end{corollary}	

A necessary condition for equality to be achieved in the estimate (\ref{iequation8}) of Theorem \ref{it2}, is given by

\begin{theorem}\label{itheorem2}
		Let $\overline M_f$ be a shrinking gradient Ricci soliton with $\overline{\Ric}_f=kg$. Let $M^n$ be a constant weighted mean curvature hypersurface isometrically immersed in  $\overline M_f$. Denote by $\mathcal P_{\overline M_f}$ the set of parallel fields globally defined on $\overline M_f^{n+1}$ and $\eta$ the unit normal field to $M$. Suppose that $\Vol_f(M)<\infty$, $\Ind_f(M)<\infty$, $\dim \mathcal P_{\overline M_f}>0$, and
	\begin{equation*}\label{ieq5}
	\Ind_f(M)=\dim \mathcal P_{\overline M_f}-\dim\{X\in P_{\overline M_f}: \ \langle X,\eta\rangle\equiv0\}. 
	\end{equation*}
	\begin{enumerate}
		\item[(i)] If $\Ind_f(M)=\dim \mathcal P_{\overline M_f}$, $M$ is totally geodesic and  the bottom $\mu_1(M)$ of the spectrum of $f$-stability operator satisfies  $\mu_1(M)=-k.$ 
		\item[(ii)] If $\Ind_f(M)\neq\dim \mathcal P_{\overline M_f},$ either $M$ is diffeomorphic to the product of a Euclidian space with some other manifold or there is a circle action on $M$ whose orbits are not real homologous to zero.
	\end{enumerate} 
\end{theorem}

\section{Properness and Finite Weighted Volume of a Constant Weighted Mean Curvature Hypersurface}

We will begin by proving Proposition \ref{ip4} which states that a noncompact complete hypersurface with finite weighted volume, weighted mean curvature vector bounded in norm, and isometrically immersed in a complete weighted manifold is proper. Furthermore, long after we will prove Proposition \ref{ip1}, it give some conditions to that the a proper noncompact complete hypersurface have finite weighted volume.

\begin{proof}[Proof of Proposition \ref{ip4}.] We supposed that $M$ is not proper. Thus, there exists a positive real number $R$ such that $\overline B_R^{\overline M}(o)\cap M$ is no compact in $M$, where $\overline B_R^{\overline M}(o)$ denotes the closure of the $B_R^{\overline M}(o)$. Then, for any $a>0$ sufficiently small with $a<2R$, there exists a sequence $\{p_k\}$ of the points in $B_R^{\overline M}(o)\cap M$ with $\dist_{M}(p_k,p_j)\geq a>0$ for any different $k$ and $j$. 
	
	Since $B_{a/2}^{M}(p_k)\cap B_{a/2}^{M}(p_j)=\emptyset$ for any $k\neq j$, we obtain $B_{a/2}^{M}(p_j)\subset B_{2R}^{\overline M}(o)$, where $B_{a/2}^{M}(p_k)$ and $B_{a/2}^{M}(p_j)$ denote the intrinsic balls of $M$ of radius $a/2$, center in $p_k$ and $p_j$, respectively. 
	
	Let $\{e_1,e_2,\ldots,e_n\}$ be a orthonormal basis of $T_xM$. If $x\in B_{a/2}^{M}(p_j)$, then the function extrinsic distance to $p_j$, denoted by $r_j(x)=\dist_{\overline M}(x,p_j)$, satisfies
	\begin{eqnarray*}
		\overline{\nabla}^2r_j(e_i,e_i)&=&\langle\overline{\nabla}_{e_i}\overline{\nabla} r_j,e_i\rangle=\langle \overline{\nabla}_{e_i}\nabla r_j,e_i\rangle+\langle \overline{\nabla}_{e_i}(\overline{\nabla}r_j)^{\perp},e_i\rangle\\
		&=&\langle {\nabla}_{e_i}\nabla r_j,e_i\rangle+\langle(\overline{\nabla}_{e_i}\nabla r_j)^{\perp},e_i\rangle-\langle A(e_i,e_i), \overline{\nabla}r_j\rangle\\
		&=&{\nabla}^2r_j(e_i,e_i)-\langle {\bf H}, \overline{\nabla}r_j\rangle.
	\end{eqnarray*}
	Observe that $\overline M$ has bounded locally geometry, this is, there exists positive real numbers $k$ and $i_0$ so that the sectional curvature of $\overline M$ is bounded above by $k$ and the injectivity radius of $\overline M$ is bounded below by $i_0$ in a neighborhood of a point $o\in \overline M$. By choosing $R>0$ such that $2R<\min\{ i_0,1/\sqrt k\}$, it follows from Hessian comparison of the distance (see Lemma 7.1, \cite{ColdingMinicozzi2011}), that
	$$\overline\nabla^2r_j(e_i,e_i)\geq-\sqrt k+\dfrac{1}{r_j}|e_i-\langle e_i,\overline{\nabla}r_j\rangle\overline{\nabla}r_j|^2$$
	in $\overline B_{2R}^{\overline M}(o).$ Hence, in $\overline B_{2R}^{\overline M}(o)\cap M$,	
	\begin{eqnarray*}
		\Delta r_j&=&\sum_{i=1}^{n}\nabla^2r_j(e_i,e_i)=\sum_{i=1}^{n}\overline\nabla^2r_j(e_i,e_i)+\langle {\bf H}, \overline{\nabla}r_j\rangle\\
		&\geq&\sum_{i=1}^{n}\left(-\sqrt k+\dfrac{1}{r_j}|e_i-\langle e_i,\overline{\nabla}r_j\rangle\overline{\nabla}r_j|^2\right)+\langle {\bf H}, \overline{\nabla}r_j\rangle+\langle (\overline{\nabla}f)^{\perp}, \overline{\nabla}r_j\rangle-\langle (\overline{\nabla}f)^{\perp}, \overline{\nabla}r_j\rangle\\	
		&=&-n\sqrt k+\dfrac{n}{r_j}-\dfrac{|\nabla r_j|^2}{r_j}+\langle {\bf H}_f, \overline{\nabla}r_j\rangle-\langle (\overline{\nabla}f)^{\perp}, \overline{\nabla}r_j\rangle\\
		&\geq&-n\sqrt k+\dfrac{n}{r_j}-\dfrac{|\nabla r_j|^2}{r_j}-|{\bf H}_f|-|(\overline\nabla f)^{\perp}|.
	\end{eqnarray*}
	
	By hypothesis, the norm of ${\bf H}_f$ is bounded above  and $\sup_{p\in\overline B_{2R}^{\overline M}(o)}|\overline\nabla f(p)|<\infty$ for each $R>0$. Thus,
	\begin{eqnarray*}
		\Delta r_j
		&\geq&-n\sqrt k+\dfrac{n}{r_j}-\dfrac{|\nabla r_j|^2}{r_j}-\sup_{p\in B_{2R}^{\overline M}(o)\cap M}|{\bf H}_f(p)|-\sup_{p\in B_{2R}^{\overline M}(o)\cap M}|(\overline\nabla f(p))^{\perp}|\\
		&\geq&\dfrac{n}{r_j}-\dfrac{|\nabla r_j|^2}{r_j}-C,
	\end{eqnarray*}
	where $$C=n\sqrt k+\sup_{p\in B_{2R}^{\overline M}(o)\cap M}|{\bf H}_f(p)|+\sup_{p\in \overline B_{2R}^{\overline M}(o)}|\overline\nabla f(p)|.$$
	Therefore, in $B_{2R}^{\overline M}(o)\cap M$,
	$$\Delta r_j^2=2r_j\Delta r_j+2|\nabla r_j|^2\geq2r_j\left(\dfrac{n}{r_j}-\dfrac{|\nabla r_j|^2}{r_j}-C\right)+2|\nabla r_j|^2
	= 2n-2C r_j.$$
	By choosing $a<\min\{2n/C,2R\},$ we have for $0<\zeta\leq a/2,$
	\begin{eqnarray}\label{e1}
	\int_{B_{\zeta}^{M}(p_j)}(2n-2C r_j)\,d\sigma&\leq&\int_{B_{\zeta}^{M}(p_j)}\Delta r^2_j\,d\sigma=\int_{\partial B_{\zeta}^{M}(p_j)}\langle\nabla r^2_j,\nu\rangle\,d\sigma\nonumber\\
	&\leq&\int_{\partial B_{\zeta}^{M}(p_j)} 2r_j|\nabla r_j||\nu| dA\leq\int_{\partial B_{\zeta}^{M}(p_j)} 2r_j\,dA\\
	&\leq& 2\zeta A_j	
	(\zeta)\nonumber,
	\end{eqnarray}
	where $\nu$ denotes the unit normal vector field pointing out of $\partial B_{\zeta}^{M}(p_j)$ and $A_j(\zeta)$ denotes the area of $\partial B_{\zeta}^{M}(p_j)$. By using co-area formula in (\ref{e1}), we have 
	\begin{eqnarray*}
		\int_{B_{\zeta}^{M}(p_j)}(n-C r_j)\,d\sigma&=&\int_0^\zeta\int_{\partial B_{t}^{M}(p_j)}(n-C r_j)|\nabla r_j|^{-1}dA_t\,dt\\	
		&\geq&\int_0^\zeta\int_{\partial B_{t}^{M}(p_j)}(n-C r_j)dA_t\,dt\\
		&\geq&\int_0^\zeta(n-C t)\int_{\partial B_{t}^{M}(p_j)}dA_t\,dt\\
		&\geq&(n-C\zeta)V_j(\zeta),
	\end{eqnarray*}
	where $V_j(\zeta)$ denotes the volume of $B_{\zeta}^{M}(p_j)$.	Therefore,
	\begin{eqnarray}
	(n-C\zeta)V_j(\zeta)\leq\zeta A_j(\zeta).
	\end{eqnarray}
	Since
	$$V_j'(\zeta)=\dfrac{d}{d\sigma}\int_{B_{\zeta}^{M}(p_j)}\,d\sigma=\dfrac{d}{d\sigma}\int_0^\zeta\int_{\partial B_{t}^{M}(p_j)}|\nabla r_j|^{-1}dA_t\,dt=\int_{\partial B_{\zeta}^{M}(p_j)}|\nabla r_j|^{-1}dA\geq A_j(\zeta),$$
	then
	$$(n-C\zeta)V_j(\zeta)\leq \zeta V_j'(\zeta).$$
	Thus,
	\begin{equation}\label{e2}
	\dfrac{d}{d\sigma}\log V_j(\zeta)=\dfrac{V_j'(\zeta)}{V_j(\zeta)}\geq\dfrac{n}{\zeta}-C.
	\end{equation}
	By integrating (\ref{e2}) from $\varepsilon>0$ to $\zeta$, we obtain
	$$\log V_j(\zeta)-\log V_j(\varepsilon)\geq n\log\zeta-n\log\varepsilon-C(\zeta-\varepsilon),$$
	that is
	$$\dfrac{V_j(\zeta)}{V_j(\varepsilon)}\geq\dfrac{\zeta^n}{\varepsilon^n}e^{-C(\zeta-\varepsilon)}.$$
	Now observing that $$\lim_{\varepsilon\rightarrow0^+}\dfrac{V_j(\varepsilon)}{\varepsilon^n}=\omega_n,$$ we obtain 
	$$V_j(\zeta)\geq\omega_n\zeta^ne^{-C\zeta}$$
	for $0<\zeta\leq a/2.$
	Thus, we conclude that
	\begin{eqnarray*}
		\Vol_f(M)&=&\int_{M}e^{-f}\,d\sigma\geq\sum_{j=1}^{\infty}\int_{B_{a/2}^{M}(p_j)}e^{-f}\,d\sigma\\
		&\geq&\left(\inf_{\overline B_{2R}^{\overline M}(o)}e^{-f}\right)\sum_{j=1}^{\infty}V_j(a/2)=+\infty.
	\end{eqnarray*}
This contradicts the assumption of the finite weighted volume of $M$. Therefore, $M^n$ is a proper hypersurface of $\overline M^m_f.$
\end{proof}

\begin{definition}
	The function $f\in C^{\infty}(\overline M)$ is said {\it convex} if the Hessian of $f$ is non-negative, this is $\overline{\nabla}^2f\geq0.$ 
\end{definition}

\begin{remark}\label{remark1} Let $\overline M_f$ be a  complete gradient Ricci soliton satisfying
	$\overline{\Ric}_f=\dfrac{1}{2} g.$ In this case, Cao and Zhou \cite{CaoZhou2010} showed that by translating $f$ 
	\begin{equation}
	\label{ae} \overline R+|\overline\nabla f|^2-f=0 \ \ \ \ \text{e}\ \ \ \ \ \overline R\geq0.
	\end{equation}	
	Thus, it follows from the equations in (\ref{ae}) that
	\begin{equation}
	\label{equ1}
	|\overline\nabla f|^2\leq f.
	\end{equation}
	In addition, there exists constants $c_1,c_2\in\R$ such that
	\begin{equation}
	\label{equa3} \dfrac{1}{4}(r(x)-c_1)^2\leq f(x)\leq \dfrac{1}{4}(r(x)+c_2)^2,
	\end{equation}
	where $r(x)=\dist_{\overline M}(x,o)$ is the distance of $x\in \overline M$ to a fixed point $o\in \overline M$. The constant $c_2$ depends only on the dimension of the manifold and $c_1$ depends on the geometry of $g$ on unit ball center in $o$ (see \cite{CaoZhou2010}, Lemma 2.1, Lemma 2.2 and Theorem 1.1).  In \cite{MunteanuWang2012}, Munteanu and Wang showed that the inequalities in (\ref{equa3}) are true only assuming that $\overline{\Ric}_f\geq\dfrac{1}{2}g$ and $|\overline{\nabla}f|^2\leq f.$
\end{remark}

\begin{proof}[Proof of Proposition \ref{ip1}.]
By hypothesis, $\overline{\Ric}_f=\dfrac{1}{2}g.$	Thus, it follows from Remark \ref{remark1} that
	$$\overline R+|\overline\nabla f|^2-f=0, \ \ \ \ \ \overline R+\overline{\Delta} f=\dfrac{m}{2}, \ \ \ \ \ \text{and} \ \ \ \ \ \overline R\geq0,$$
	and we have that
	$$\overline{\Delta}f-|\overline{\nabla} f|^2+f=\dfrac{m}{2} \ \ \ \ \text{and} \ \ \ \ |\overline\nabla f|^2\leq f.$$

By being $f$ a convex function, we obtain 
	\begin{eqnarray*}
		\Delta_ff+f&=&\Delta f-|\nabla f|^2+f=\overline\nabla^2f(e_i,e_i)+\langle{\bf H},\overline\nabla f^{\perp}\rangle-|\nabla f|^2+f\\
		&=&\overline{\Delta}f-\sum_{i=n+1}^m\overline{\nabla}^2 f(\eta_i,\eta_i)+\langle{\bf H}_f,\overline\nabla f^{\perp}\rangle-|\overline{\nabla}f^{\perp}|^2-|\nabla f|^2+f\\
		&=&
		\overline{\Delta}f-\sum_{i=n+1}^m\overline{\nabla}^2 f(\eta_i,\eta_i)+\langle{\bf H}_f,\overline\nabla f^{\perp}\rangle-|\overline{\nabla}f|^2+f\\
		&\leq&\dfrac{m}{2}+C,\\
	\end{eqnarray*}      	
	where $C=\sup_{x\in M}\langle{\bf H}_f,\overline\nabla f^{\perp}\rangle<\infty.$  
	
	Observe that 
	\begin{equation}\label{equa33}
	\dfrac{1}{4}(r(x)-c)^2\leq f(x)\leq \dfrac{1}{4}(r(x)+c)^2,
	\end{equation}
	where $c$ is a constant (see Remark \ref{remark1}, inequalities in (\ref{equa3})). Hence, we can conclude that $f$ is proper on $\overline M$. Since, by hypothesis, $x:M\rightarrow\overline M_f$ is a proper immersion, then $f|_M:M\rightarrow\R$ is a proper function. 
	
	Therefore, it follows from Theorem 1.1 de \cite{ChengZhou2013} that $M$ has finite weighted volume and Euclidean volume growth of the sub-level set of the potential function $f$.
\end{proof}

\section{The $f$-stability Index of the  Constant Weighted Mean Curvature Hypersurfaces}

In this section, we will prove Theorem \ref{it2}, Corollary \ref{ic1}, and Theorem \ref{itheorem2}. Which are results about the $f$-stability index of a constant weighted mean curvature hypersurface with finite weighted volume and isometrically immersed in a gradient Ricci soliton that admits at least one parallel field globally defined. For this, we are going to give some definitions and known results.

\begin{definition}
	We say that a vector field $X\in T\overline M$ is {\it parallel} if $$\overline\nabla_YX=0$$ for all vector fields $Y\in T\overline M.$  
\end{definition}

\begin{proposition}\label{proposition1}
	Let $\overline M_f^{n+1}$ be a  gradient Ricci soliton satisfying $\overline{\Ric}_f=kg$ and $X$ a parallel vector field on $\overline M_f^{n+1}$. If $M^n$ is a constant weighted mean curvature hypersurface immersed isometrically immersed  in  $\overline M_f^{n+1},$ then
	$$L_f\langle X,\eta\rangle=k\langle X,\eta\rangle$$
	and
	$$\Delta_f\langle X,\eta\rangle^{2}=-2|A|^{2}\langle X,\eta\rangle^{2}+2| AX^{\top}|^2,$$
	where $\eta$ is the unit normal vector field to $M$.	
\end{proposition}	

\begin{proof}
	Let $\{e_1,e_2,\ldots,e_n\}$ be a geodesic orthonormal frame on $M$. By hypothesis, $H_f=C$, this is $H=\langle\overline{\nabla }f,\eta\rangle+C,$ where $C$ is a real constant. Thus,
	\begin{eqnarray}\label{equation12}
		\nabla H&=&\sum_{i=1}^ne_i(H)e_i=\sum_{i=1}^ne_i( \langle \overline{\nabla}f,\eta\rangle) e_i=\sum_{i=1}^n\langle \overline{\nabla}_{e_i}\overline{\nabla}f,\eta\rangle e_i+\sum_{i=1}^n\langle\overline{\nabla }f,\overline{\nabla }_{e_i}\eta\rangle e_i\nonumber\\
		&=&\sum_{i=1}^n\langle \overline{\nabla}_{e_i}\overline{\nabla}f,\eta\rangle e_i-\sum_{i=1}^n\langle A(e_i,e_j),\eta\rangle\langle\overline{\nabla }f,e_j\rangle e_i.
	\end{eqnarray}
	To $u=\langle X,\eta\rangle$ and $a_{ij}=\langle Ae_i,e_j\rangle$, we have 
	\begin{equation}\label{equation11}
	\nabla u=\sum_{j=1}^ne_j(u)e_j=\sum_{j=1}^n\langle\overline{\nabla }_{e_j}\eta,X\rangle e_j=-\sum_{i,j=1}^na_{ji}\langle e_i,X\rangle e_j.
	\end{equation}
	It follows from (\ref{equation12}) and (\ref{equation11}),
	$$\langle \nabla H,X\rangle=\sum_{i=1}^n\langle \overline{\nabla}_{e_i}\overline{\nabla}f,\eta\rangle \langle e_i,X\rangle-\sum_{i,j=1}^na_{ij}\langle\overline{\nabla }f,e_j\rangle \langle e_i,X\rangle=\langle \overline{\nabla}_{ {X^{\top}}}\overline{\nabla}f,\eta\rangle +\langle\overline{\nabla}f,\nabla u\rangle.$$
	Moreover, 
	$$e_i(u)=\langle\overline{\nabla }_{e_i}\eta,X\rangle=-\sum_{j=1}^na_{ij}\langle e_j,X\rangle$$
	and, by deriving the previous expression and observing that $\nabla_{e_k}e_j=0$, we obtain
	$$e_k(e_i(u))=-\sum_{j=1}^n\left(a_{ij,k}\langle e_j,X\rangle+a_{ij}\langle X,\overline{\nabla}_{e_k}e_j\rangle\right)=-\sum_{j=1}^na_{ij,k}\langle e_j,X\rangle-\sum_{j=1}^na_{ij}a_{kj}\langle X,\eta\rangle.$$
	It follows from Codazzi equation that
	$$\overline R(e_j,e_k)e_i^{\perp}=\left(a_{ki,j}-a_{ji,k}\right)\eta,$$
	that is
	$$\langle \overline R(e_j,e_k)e_i,\eta\rangle=a_{ki,j}-a_{ji,k}.$$
	Hence,
	\begin{eqnarray*}
		e_k(e_i(u))&=&-\sum_{j=1}^na_{ki,j}\langle e_j,X\rangle+\sum_{j=1}^n\langle \overline R(e_j,e_k)e_i,\eta\rangle\langle e_j,X\rangle-\sum_{j=1}^na_{ij}a_{kj}\langle X,\eta\rangle
	\end{eqnarray*}
	and
	\begin{eqnarray}\label{equation13}
		\Delta u&=&\sum_{i=1}^ne_i(e_i(u))=-\sum_{i,j=1}^na_{ii,j}\langle e_j,X\rangle+\sum_{i,j=1}^n\langle \overline R(e_j,e_i)e_i,\eta\rangle\langle e_j,X\rangle-\sum_{i,j=1}^na_{ij}a_{ij}\langle X,\eta\rangle\nonumber\\
		&=&\langle\nabla H,X\rangle+\sum_{i=1}^n\langle\overline R(X^{\top},e_i)e_i,\eta\rangle-|A|^{2}\langle X,\eta\rangle\nonumber\\
		&=&\langle \overline{\nabla}_{ {X^{\top}}}\overline{\nabla}f,\eta\rangle +\langle\overline{\nabla}f,\nabla u\rangle+\overline{Ric}(X^{\top},\eta)-|A|^{2}u.
	\end{eqnarray}
	On the other hand,
		\begin{equation}\label{equation14}0=\dfrac{1}{2}\langle X^{\top},\eta\rangle=\overline{Ric}_f(X^{\top},\eta)=\overline{Ric}(X^{\top},\eta)+\overline{\nabla}^{2}f(X^{\top},\eta)=\overline{Ric}(X^{\top},\eta)+\langle \overline{\nabla}_{ {X^{\top}}}\overline{\nabla}f,\eta\rangle.
	\end{equation}
	Therefore, it follows from (\ref{equation13}) and (\ref{equation14}),
	\begin{eqnarray}\label{equ11}
	\Delta_fu&=&\Delta u-\langle\nabla f,\nabla u\rangle\nonumber\\
	&=&\langle \overline{\nabla}_{ {X^{\top}}}\overline{\nabla}f,\eta\rangle +\langle\overline{\nabla}f,\nabla u\rangle-\langle \overline{\nabla}_{ {X^{\top}}}\overline{\nabla}f,\eta\rangle-|A|^{2}u-\langle\nabla f,\nabla u\rangle\nonumber\\
	&=&-|A|^{2}u,
	\end{eqnarray}
by implying that
	$$ L_fu=\Delta_fu+|A|^2u+ku=ku.	$$	
	Moreover, by using the equality (\ref{equ11}), we have that
	$$\Delta_fu^{2}=2u\Delta_fu+2|\nabla u|^{2}=-2|A|^{2}u^{2}+2|AX^{\top}|^{2}.$$
\end{proof}

Let $\mathcal P_{\overline M}$ be the set of all tangent vector fields to $\overline M$ which are 
parallel and globally defined.

\begin{example}
	O shrinking Gaussian soliton $\left(\R^{n+1},g_{can},e^{-|x|^2/2}\right)$ has exactly $n+1$ 
	parallel vector fields linearly independent and globally defined on $\R^{n+1}$.	Other example is the shrinking cylinder soliton    $\left(\s^{n+1-k}\times\R^k,g,e^{-f}\right)$, $k\geq1$, with metric $$g=2(n-k-1)g_{\s^{n-k}}+g_{\R^{k}}$$ and potential function $$f(\theta,x)=\frac{|x|^2}{2}, \ \theta\in \s^{n-k}, \ x\in\R^k.$$ In this example, we have that $$\dim\mathcal P_{\s^{n+1-k}\times\R^k}=k.$$
\end{example}

\begin{definition}
	The vector subspace of $C^{\infty}(M)$ {\it generated} by $E\subset C^{\infty}(M)$, denoted by $\Span E,$ is the set of all the linear combinations of the elements of $E$.
\end{definition}

\begin{lemma}\label{lema2} Let $\overline M_f^{n+1}$ be a shrinking gradient Ricci soliton satisfying $\overline{\Ric}_f=kg$ and $M^n$ be a constant weighted mean curvature hypersurface isometrically immersed  in  $\overline M_f^{n+1}.$ If $M$ is compact, then $I_f$ is negative defined in the $$\Span\{1,\langle X,\eta\rangle: \ X \in\mathcal P_{\overline M_f}\}.$$
Moreover,
$$\int_M|A|^2\langle X,\eta\rangle e^{-f}d\sigma=0.$$	
\end{lemma}	

\begin{proof}
	It follows from Proposition \ref{proposition1} that the function $u=\langle X,\eta\rangle$, with $X\in\mathcal P_{\overline M_f}$, satisfies $L_fu=ku.$ Thus, since $M$ is compact, we have 	
	\begin{eqnarray*}
		\int_Mkue^{-f}d\sigma&=&\int_ML_fue^{-f}d\sigma=\int_M\left(\Delta_fu+|A|^2u+ku\right)e^{-f}d\sigma\\
		&=&\int_M|A|^2ue^{-f}d\sigma+\int_Mkue^{-f}d\sigma.
	\end{eqnarray*}
	Therefore,
	\begin{equation}\label{eeee7}
	\int_M|A|^2ue^{-f}d\sigma=0.
	\end{equation}	
	Observe that	
	\begin{equation*}\label{eeee8}
	I_f(1,1)=-\int_M1L_f1e^{-f}d\sigma=-\int_M\left(|A|^2+k\right)e^{-f}d\sigma
	\end{equation*}	
	and
	\begin{equation*}\label{eeee9}I_f(u,u)=-\int_MuL_fue^{-f}\,d\sigma=-k\int_Mu^2e^{-f}\,d\sigma.
	\end{equation*}	
	Therefore,
	\begin{eqnarray*}
		I_f(c_0+u,c_0+u)&=& I_f(c_0,c_0)+I_f(u,u)+2I_f(c_0,u)\\
		&=& -\int_M\left[c_0^2|A|^2+kc_0^2+ku^2+2c_0\left(\Delta_fu+|A|^2u+ku\right)\right]e^{-f}d\sigma\\
		&=& -\int_M\left[c_0^2|A|^2+kc_0^2+ku^2+2c_0ku\right]e^{-f}d\sigma\\
		&=& -c_0^2\int_M|A|^2e^{-f}d\sigma-k\int_M(c_0+u)^{2}e^{-f}d\sigma<0,
	\end{eqnarray*}	
	where $u=\langle X,\eta\rangle.$ This shows that $I_f$ is negative defined in $\Span\{1,\langle X,\eta\rangle: \ X\in\mathcal P_{\overline M_f}\}.$
\end{proof}

\begin{remark}
	Supposing that $M^n$ has finite weighted volume and putting $\alpha=-\int_M\langle X,\eta\rangle e^{-f}d\sigma,$ we can conclude that   
	$$\int_M(\alpha+\langle X,\eta\rangle )e^{-f}d\sigma=0.$$
	Therefore, $$\mathcal F\cap\Span\{1,\langle X,\eta\rangle: X \in\mathcal P_{\overline M_f}\}\not=\emptyset.$$ 
\end{remark}

Now, let's look at the noncompact manifolds. For this, we will consider the functions that have compact support in $M$. 

\begin{proposition}\label{proposition3} Let $\overline M_f^{n+1}$ be a weighted manifold and $M^n$ be a noncompact hypersurface isometrically immersed  in  $\overline M_f^{n+1}.$ Then
	$$I_f(\phi u,\phi u)=-\int_M\phi^2uL_fue^{-f}d\sigma+\int_M|\nabla\phi|^2u^2e^{-f}d\sigma,$$
where  $\phi\in C_c^{\infty}(M)$, $u\in C^{\infty}(M)$, and $\eta$ denotes the unit normal field on $M$. Moreover, if $\overline{\Ric}_f=kg$ and $M$ has constant weighted mean curvature, then
	\begin{equation}\label{equation7}\int_M\phi^2|A|^2\langle X,\eta\rangle e^{-f}d\sigma=-2\int_M\phi \langle\nabla\phi,AX^{\top}\rangle e^{-f}d\sigma,
	\end{equation}
	where $X$ is a parallel vector field on $\overline M_f^{n+1}$.
\end{proposition}

\begin{proof}
	Note that
	\begin{eqnarray*}
		I_f(\phi u,\phi u)&=&-\int_M(\phi u)L_f(\phi u)e^{-f}d\sigma\\
		&=&-\int_M\left[(\phi u)\Delta_f(\phi u)+\left(|A|^2+\overline{\Ric}_f(\eta,\eta)\right)\phi^2 u^2\right]e^{-f}d\sigma\\
		&=&-\int_M\left[\phi^2u\Delta_fu+\phi u^2\Delta_f\phi+2\phi u\langle\nabla\phi,\nabla u\rangle+\left(|A|^2+\overline{\Ric}_f(\eta,\eta)\right)\phi^2 u^2\right]e^{-f}d\sigma,
	\end{eqnarray*}
As $\phi$ has compact support, then
$$0=\int_M\di(\phi u^2e^{-f}\nabla\phi)d\sigma=\int_M\left(\phi u^2\Delta_f\phi+2u\phi\langle\nabla u,\nabla\phi\rangle+u^2|\nabla\phi|^2\right)e^{-f}d\sigma.$$
	Therefore, by using the last two expressions, we obtain
	\begin{eqnarray*}
		I_f(\phi u,\phi u)&=&-\int_M\left[\phi^2u\Delta_fu-|\nabla\phi|^2u^2+\left(|A|^2+\overline{\Ric}_f(\eta,\eta)\right)\phi^2 u^2\right]e^{-f}d\sigma\\
		&=&-\int_M\phi^2uL_fue^{-f}d\sigma+\int_M|\nabla\phi|^2u^2e^{-f}d\sigma.
	\end{eqnarray*}
	Moreover, supposing $\overline{\Ric}_f=kg$, $H_f$ is constant, and $X$ is a parallel field on $\overline M_f$, we obtain 
	\begin{eqnarray*}
		\int_M\phi^2k\langle X,\eta\rangle e^{-f}d\sigma&=&\int_M\phi^2 L_f\langle X,\eta\rangle e^{-f}d\sigma\\
		&=&-\int_M\langle \nabla\phi^2,\nabla\langle X,\eta\rangle\rangle e^{-f}d\sigma+\int_M\left(|A|^2+k\right)\phi^2\langle X,\eta\rangle e^{-f}d\sigma.
	\end{eqnarray*}
	Hence, 
	$$\int_M|A|^2\phi^2\langle X,\eta\rangle e^{-f}d\sigma=\int_M\langle \nabla\phi^2,\nabla\langle X,\eta\rangle\rangle e^{-f}d\sigma=-2\int_M \phi \langle\nabla\phi,AX^{\top}\rangle e^{-f}d\sigma.$$
\end{proof}

\begin{lemma} \label{lema5} 	
Let $\overline M_f^{n+1}$ be a shrinking gradient Ricci soliton satisfying $\overline{\Ric}_f=kg$ and $M^n$ be a constant weighted mean curvature hypersurface isometrically immersed  in  $\overline M_f^{n+1}.$ If $M$ is noncompact and finite weighted volume, there exists $\phi\in C^{\infty}_c(M)$ such that $I_f$ is negative defined in $\phi V$ and $\dim(\phi V)=\dim V$, where
	$$V={\rm \Span}\left\{1,\langle X,\eta\rangle: \ X \ \text{is a parallel field on }\overline M^{n+1}_f\right\}.$$  
Moreover, by assuming that $\int_M|A|^2e^{-f}d\sigma<\infty$, we have
$$\int_M|A|^2\langle X,\eta\rangle e^{-f}d\sigma=0.$$		
\end{lemma}	

\begin{proof}
Let $u=c_0+\langle X,\eta\rangle$, where $c_0$ is a real constant and $X$ is a parallel field on $\overline M^{n+1}_f$. As $H_f$ is constant, then by Proposition \ref{proposition1}, we have
	\begin{eqnarray}\label{equation15}
		L_f u&=&\Delta_f u+\left(|A|^2+k\right)u=L_f\langle X,\eta\rangle+\left(|A|^2+k\right)c_0\nonumber\\
		&=&k\langle X,\eta\rangle+\left(|A|^2+k\right)c_0.
	\end{eqnarray}
	It follows from Proposition \ref{proposition3} and equality (\ref{equation15}) that 
	\begin{eqnarray*}
		I_f(\phi u,\phi u)&=&-\int_{M}\phi^2uL_fue^{-f} d\sigma+\int_{M}|\nabla\phi|^2u^2 e^{-f}d\sigma\\
		&=&-\int_{M}\phi^2u\left(k\langle X,\eta\rangle+\left(|A|^2+k\right)c_0\right) e^{-f}d\sigma+\int_{M}|\nabla\phi|^2u^2 e^{-f}d\sigma\\
		&=&-k\int_{M}\phi^2u^2e^{-f}d\sigma-\int_{M}\phi^2|A|^2c_0^2e^{-f}d\sigma-\int_{M}\phi^2|A|^2c_0\langle X,\eta\rangle e^{-f}d\sigma\\
		&&+\int_{M}|\nabla\phi|^2u^2 e^{-f}d\sigma.
	\end{eqnarray*}	
	Now, by using once more the Proposition \ref{proposition3} and Cauchy-Schwarz's inequality, 
	\begin{eqnarray*}
		\left|\int_M\phi^2|A|^2c_0\langle X,\eta\rangle e^{-f}d\sigma\right|&=& 2\left|\int_M\phi \langle\nabla\phi, AX^{\top}\rangle c_0  \ e^{-f}d\sigma\right|\\
		&\leq& 2\int_M|\phi|\,| \nabla\phi|\,|A|\,|X^{\top}|\, |c_0|  \ e^{-f}d\sigma\\
		&\leq&\int_M\phi^2|A|^2 c_0^2\, e^{-f}d\sigma+\int_M|\nabla\phi|^2|X^{\top}|^2e^{-f}d\sigma.
	\end{eqnarray*}
	Therefore,
	\begin{equation}
	\label{eq6}
	I_f(\phi u,\phi u)\leq-k\int_M\phi^2 u^2e^{-f}d\sigma+\int_M|\nabla \phi|^2(u^2+|X^{\top}|^2)e^{-f}d\sigma.
	\end{equation}
	Let $r(x)$ be the extrinsic distance from $x\in M$ to a fixed point $o\in \overline M_f$. For $R>0$ sufficiently large, define the function $\phi_R:M\rightarrow\R$ such that
	\begin{equation}
	\label{equation5}
	\phi_R(x)=\left\{
	\begin{array}{lcl}
	1, & r(x)\leq R;\\
	\dfrac{2R-r(x)}{R}, & R\leq r(x)\leq2R;\\
	0,& r(x)\geq2R.
	\end{array}
	\right.
	\end{equation}
	Observe that $|\nabla\phi_R|\leq1/R$ and $\phi_R\in C^{\infty}_c(M)$ because $M$ is proper (see Proposition \ref{ip4}). Now, by substituting $\phi=\phi_R$ in the inequality (\ref{eq6}), we obtain 
	$$I_f(\phi_R u,\phi_R u)\leq-k\int_M\phi^2_R u^2e^{-f}d\sigma+\dfrac{|X|^2+c_0^2+2|c_0|\,|X|}{R^2}\int_{M\cap(B_{2R}\setminus B_R)}e^{-f}d\sigma,$$
	(recall that $|X|$ is constant because $X$ is parallel). Since $\Vol_f(M)<\infty$ and $|X|$ is constant, then for each $u\in V$ there exists $R_u$ sufficiently large such that 
	$$I_f(\phi_{R_u}u,\phi_{R_u}u)<0.$$ 
	
	Let's find a function $\phi\in C^{\infty}_c(M)$ that is not dependent of the function $u$. In fact, we consider the subset $$S=\left\{u\in V:  \ \int_Mu^2e^{-f}d\sigma=1\right\}.$$
	Note that $V\subset L_f^2(M)$ is a subspace of finite dimension smaller than or equal to $\dim \mathcal P_{\overline M_f}+1$. Hence, $S$ is a compact set in $L_f^2(M)$. Thus, there exists a positive real number $R_0$ such that any function $u\in S$ vanishes identically on $M\cap B_{R_0}^{\overline M_f}(o)$. Otherwise, we could get a sequence ${R}_j\rightarrow\infty$ of positive numbers so that for each $j$ exists $u_j\in S$ with $u_j\equiv0$ on $M\cap B_{R_0}^{\overline M_f}(o)$. Hence, we would have $$u=\lim_{j\rightarrow\infty}u_j \ \in \ S$$ and \ $u\equiv0$ on $M.$ However, this is not possible because if $u\in S$, then $$\int_M u^2e^{-f}d\sigma=1.$$
	Therefore, to $R$ sufficiently large and $R\geq R_0$, and for any function $u\in S$, we have
	$$I_f(\phi_R u,\phi_R u)\leq-k\int_M\phi^2_R u^2e^{-f}d\sigma+\dfrac{|X|^2+c_0^2+2|c_0||X|}{R^2}\int_{M\cap(B_{2R}\setminus B_R)}e^{-f}d\sigma<0$$
	because 
	$$M(R)=\int_M\phi_R^2u^2e^{-f}d\sigma >0$$
	is an increasing function on $R$, and as $\Vol_f(M)<\infty$, 
	$$\lim_{R\rightarrow\infty}\dfrac{|X|^2+c_0^2+2|c_0||X|}{R^2}\int_{M\cap(B_{2R}\setminus B_R)}e^{-f}d\sigma=0.$$
	Now, we can find $\phi=\phi_R$ independent of $u$ such that $I_f(\phi u,\phi u)<0$ for all $u\in S$. If $u\in V$, then $\dfrac{1}{|u|_{L^{2}_f}}u\in S$, and thus, for $u\not\equiv0$,
	$$I_f(\phi u,\phi u)=|u|^{2}_{L^{2}_f}I_f\left(\phi \dfrac{u}{|u|_{L^{2}_f}},\phi \dfrac{u}{|u|_{L^{2}_f}}\right)<0.$$
	
	Now, we will show that $\dim V=\dim(\phi V).$ In fact, let $\{u_1,u_2,\ldots,u_s\}$ be a orthonormal basis to the vector subspace $V\subset L^{2}_f(M)$. For the function $\phi$ built here, we have that $u_i\not\equiv0$ on $M\cap B_R^{\overline M_f}(o).$ Therefore, $\{\phi u_1,\phi u_2,\ldots,\phi u_s\}$ is linearly independent, so we can conclude that $\dim(\phi V)=\dim V.$
	
	Finally, it follows from Proposition \ref{proposition3},	equation (\ref{equation7}) with $\phi=\sqrt\phi_j$, and Cauchy–Schwarz inequality that 
	\begin{eqnarray*}\label{equation9}\left|2\int_M\phi_j|A|^2\langle X,\eta\rangle e^{-f}d\sigma\right|&=&\left|-2\int_M \langle\nabla\phi_j,AX^{\top}\rangle e^{-f}d\sigma\right|\\
		&\leq&\left(\int_M|\nabla \phi_j|^2e^{-f}d\sigma\right)^{\frac{1}{2}}\left(\int_M |AX^{\top}|^2e^{-f}d\sigma\right)^{\frac{1}{2}}.
	\end{eqnarray*}
	Now, putting $R=j$ in the function defined in (\ref{equation5}) and reviewing that $|X|$ is constant, we get
	\begin{equation}\label{equation6}\left|2\int_M\phi_j|A|^2\langle X,\eta\rangle e^{-f}d\sigma\right|
		\leq\frac{|X|}{j}\left(\int_{M\cap(B_{2j}\setminus B_j)}e^{-f}d\sigma\right)^{\frac{1}{2}}\left(\int_M |A|^2e^{-f}d\sigma\right)^{\frac{1}{2}}.
	\end{equation}
	By hypothesis $M$ has finite weighted volume and $\int_M|A|^2e^{-f}d\sigma<\infty,$	hence the right-hand side of (\ref{equation6}) tends to zero as $j\rightarrow\infty$, and we can conclude that
	\begin{equation}\label{equation8}
	\lim_{j\rightarrow\infty}\int_M\phi_j|A|^2\langle X,\eta\rangle e^{-f}d\sigma=0.
	\end{equation}	
	Since $$\lim_{j\rightarrow\infty}(\phi_j|A|^2\langle X,\eta\rangle)(x)=(|A|^2\langle X,\eta\rangle)(x)$$ for each $x\in M$, $$|\phi_j|A|^2\langle X,\eta\rangle|\leq|A|^2|\langle X,\eta\rangle|$$
	and
	$$0\leq\int_M|A|^2|\langle X,\eta\rangle| e^{-f}d\sigma\leq|X|^2\int_M|A|^2e^{-f}d\sigma<\infty,$$
	then using the dominated convergence theorem and expression (\ref{equation8}), we get
	$$\int_M|A|^2\langle X,\eta\rangle e^{-f}d\sigma=0.$$
\end{proof}

\begin{proof}[Proof of Theorem \ref{it2}.]
	Let $V\equiv \Span\{1,\langle X,\eta\rangle: \ X\in\mathcal P_{\overline M_f}\}$. Since $1\not\in \{\langle X,\eta\rangle: \ {X\in\mathcal P_{\overline M_f}}\}$ by hypothesis, 
	$$\dim V=1+\dim \{\langle X,\eta\rangle: \ {X\in\mathcal P_{\overline M_f}}\}.$$
	By Lemmas \ref{lema2} and \ref{lema5}, there exists a function $\phi\in C_c^{\infty}(M)$ such that $\dim\phi V=\dim V$ and $I_f$ is negative defined in $\phi V.$ In compact hypersurfaces, choose $\phi\equiv1$. Recall that the $\Ind_f(M)$ is the maximal dimension of a subspace of $\mathcal F\cap C_c^{\infty}(M)$ which $I_f$ is negative defined, with
	$\mathcal F=C_c^{\infty}(M)$ if $M$ is a $f$-minimal and $$\mathcal F=\left\{u\in C^{\infty}_c(M); \ \int_Mue^{-f}d\sigma=0\right\}$$
	if $M$ has constant weighted mean curvature. 
	
	Now, observe that $$\dim\{\langle X,\eta\rangle: \ {X\in\mathcal P_{\overline M_f}}\}\leq\dim (\mathcal F\cap \phi V).$$
	In fact, $|X|$ is constant because $X$ is a parallel field, and since the weighted volume of $M$ is finite, we obtain that $\int_M\phi\langle X,\eta\rangle e^{-f}d\sigma\leq|X|\int_M\phi e^{-f}d\sigma<\infty$. Hence, there exists a real number $c_0$ satisfying 
	$$\int_M\phi(c_0+\langle X,\eta\rangle)e^{-f}d\sigma=0.$$
	Therefore, $\phi(c_0+\langle X,\eta\rangle)\in \mathcal F\cap \phi V$ whenever $\langle X,\eta\rangle\in\{\langle X,\eta\rangle: \ {X\in\mathcal P_{\overline M_f}}\}$. Since $I_f$ is negative defined in $\mathcal F\cap \phi V$, we conclude that 
	$$\dim\{\langle X,\eta\rangle: \ {X\in\mathcal P_{\overline M_f}}\}\leq\dim (\mathcal F\cap \phi V)\leq\Ind_f(M).$$
	
	Consider the linear transformation $T:\mathcal P_{\overline M_f}\rightarrow C^{\infty}(M)$ defined by $T(X)=\langle X,\eta\rangle$. Now applying the kernel and image theorem, that turns
	\begin{eqnarray}\label{b}
	\dim \mathcal P_{\overline M_f}&=&\dim\{X\in\mathcal P_{\overline M_f}: \ \langle X,\eta\rangle\equiv0\}+\dim\{\langle X,\eta\rangle; \ {X\in\mathcal P_{\overline M_f}}\}\nonumber\\
	&\leq&\dim\{X\in\mathcal P_{\overline M_f}: \ \langle X,\eta\rangle\equiv0\}+\Ind_f(M).
	\end{eqnarray}
	Otherwise, if $1\in\{\langle X,\eta\rangle: \ {X\in\mathcal P_{\overline M_f}}\},$  $L_f1=k$ by Proposition \ref{proposition1}, and by definition $L_f1=|A|^2+k$. Therefore, $|A|^2\equiv0.$
\end{proof}

\begin{proof}[Proof of Corollary \ref{ic1}.]
	Note that 
	\begin{equation}
	\label{a} \dim\{X\in\mathcal P_{\overline M_f}: \ \langle X,\eta\rangle\equiv0\}\leq \dim\mathcal P_{\overline M_f}-1
	\end{equation}
	as long as we assume that exists a field $X_0\in \mathcal P_{\overline M_f}$ such that $\langle X_0,\eta\rangle\not\equiv0$.
	It follows from Theorem \ref{it2} and from the inequality (\ref{a}) that 
	$$\dim\mathcal P_{\overline M_f}-\Ind_f M\leq \ \dim\{X\in\mathcal P_{\overline M_f}: \ \langle X,\eta\rangle\equiv0\} \ \leq \dim\mathcal P_{\overline M_f}-1.$$	
	Therefore, $$\Ind_f M\geq1.$$
	Now supposing $\Ind_f M=1$, we have
	$$\dim\mathcal P_{\overline M_f}-1\leq \dim\{X\in \mathcal P_{\overline M_f}: \ \langle X,\eta\rangle\equiv0\}\leq \dim\mathcal P_{\overline M_f}-1.$$
	Therefore,
	$$\dim\{X\in \mathcal P_{\overline M_f}: \ \langle X,\eta\rangle\equiv0\}=\dim\mathcal P_{\overline M_f}-1.$$
\end{proof}

Now, let's obtain a necessary condition for a constant weighted mean curvature hypersurface $M^n$ with finite weighted volume satisfies the following equality:  
	 	\begin{equation*}
	 	\Ind_f(M)=\dim \mathcal P_{\overline M_f}-\dim\{X\in P_{\overline M_f}: \ \langle X,\eta\rangle\equiv0\}. 
	 	\end{equation*}
For this, we will prove some lemmas, they are adaption of known results.

The weighted Laplacian operator $\Delta_f$ is associated to $e^{-f}d\sigma$ as well as $\Delta$ is associated to $d\sigma.$ This is viewed in the following Green's theorem version for wighted Laplacian:

\begin{lemma}\label{Green} Let $\Omega\subset M$ be a compact set and let $u,v\in C^{\infty}(\Omega)$, then
	\begin{equation}\label{G1}
	\int_{\Omega}(u\Delta_fv)e^{-f}\,d\sigma+\int_{\Omega}\langle\nabla u,\nabla v\rangle e^{-f}\,d\sigma=\int_{\partial \Omega}u\nu(v) e^{-f}\,d\partial \Omega, \ \ \ \text{and}
	\end{equation}
	\begin{equation}\label{G2}
	\int_{\Omega}(u\Delta_fv-v\Delta_fu)e^{-f}\,d\sigma=\int_{\partial \Omega}(u\nu(v)-v\nu(u)) e^{-f}\,d\partial \Omega,
	\end{equation}
	where $\nu$ denotes the exterior unit normal to $\Omega$ along  $\partial\Omega$.
\end{lemma}
\begin{proof}
	In fact,
	\begin{eqnarray*}
		\di(e^{-f}u\nabla v)&=& e^{-f}u\,\di(\nabla v)+\langle\nabla(e^{-f}u),\nabla v\rangle\\
		&=& e^{-f}u\Delta v-e^{-f}u\langle\nabla f,\nabla v\rangle+e^{-f}\langle\nabla u,\nabla v\rangle\\
		&=& e^{-f}u\Delta_f v+e^{-f}\langle\nabla u,\nabla v\rangle.
	\end{eqnarray*}
	Integrating both sides from above inequality and applying the divergent theorem to the field $X=e^{-f}u\nabla v$, it follows the equality (\ref{G1}). The equality  (\ref{G2}) is obtained by integrating  the difference $u\Delta_f v-v\Delta_f u$ and by applying the equality (\ref{G1}). 
\end{proof}

\begin{lemma}[\cite{ColdingMinicozzi2012}, Corollary 3.10]\label{lemma2}
	Suppose that $M$ is a complete hypersurface without boundary. If $u,$ $v$ are $C^2$ functions with
	\begin{equation}\label{equation1}
	\int_M\left(|u\nabla v|+|\nabla u|\,|\nabla v|+|u\Delta_f v|\right)e^{-f}d\sigma<\infty,
	\end{equation}  
	then 
	\begin{equation}
	\int_M u(\Delta_f v)e^{-f}d\sigma=-\int_M\langle\nabla v,\nabla u\rangle e^{-f}d\sigma.
	\end{equation}
	\end{lemma}

Let $W^{1,2}(e^{-f}d\sigma)$ be the weighted Sobolev space, which is the set of the functions $u$ on $M$ satisfying  $$\int_M(u^2+|\nabla u|^2)e^{-f}d\sigma<\infty$$
with the norm
$$\|u\|_{W^{1,2}_f}:=\left(\int_M(u^2+|\nabla u|^2)e^{-f}d\sigma\right)^{\frac{1}{2}}.$$
Fixed the notation $W^{1,2}_f=W^{1,2}_f(e^{-f}d\sigma)$ and $L^{2}_f=L^{2}_f(e^{-f}d\sigma)$, we get

\begin{lemma} \label{l2} Let $M^n$ be a constant weighted mean curvature hypersurface isometrically immersed in a gradient Ricci soliton $\overline M_f$ that satisfies $\overline{\Ric}_f=kg$.	Suppose that $h$ is a $C^2$ function with $L_fh=-\mu h$ for  $\mu\in\R.$
	\begin{enumerate}
	\item[(i)] If $h\in W^{1,2}_f$, then $|A|h\in L^2_f$ and 
	\begin{equation}\label{equation0}
	\int_M|A|^2h^2e^{-f}d\sigma\leq\int_M\left((1-k-\mu)h^2+2|\nabla h|^2\right)e^{-f}d\sigma.
	\end{equation}
	\item[(ii)] If $h>0$ and $\phi\in W^{1,2}_f,$ then
	\begin{equation}\label{equation3}
	\int_M\phi^2\left(2|A|^2+|\nabla\log h|^2\right)e^{-f}d\sigma\leq\int_M\left[4|\nabla\phi|^2-2(\mu+k)\phi^2\right]e^{-f}d\sigma.
	\end{equation}
	\end{enumerate}
\end{lemma}

\begin{proof}
	Let $\phi$ be a smooth function with compact support. Note that
	$$\Delta_f h^2=2|\nabla h|^2+2h\Delta_f h$$ and 
	\begin{equation}\label{equation19}
	\Delta_f h=\left(L_f-|A|^2-k\right)h=-\left(\mu+|A|^2+k\right)h.
	\end{equation}
	 By Lemma \ref{Green} and equality (\ref{equation19}),
	\begin{eqnarray}\label{equation2}
	\int_M\langle\nabla\phi^2,\nabla h^2\rangle e^{-f}d\sigma&=&-\int_M\phi^2\Delta_f h^2 e^{-f}d\sigma\nonumber\\
	&=&-2\int_M\phi^2\left[|\nabla h|^2-\left(\mu+|A|^2+k\right)h^2\right]e^{-f}d\sigma.
	\end{eqnarray} 
	Assume now that $\phi\leq1$ and $|\nabla\phi|\leq1$. Rearranging the terms in (\ref{equation2}) and using the inequality
	$$0\leq|\phi\nabla h-h\nabla\phi|^2=\phi^2|\nabla h|^2+h^2|\nabla\phi|^2-2\phi h\langle\nabla\phi,\nabla h\rangle\leq |\nabla h|^2+h^2-2\phi h\langle\nabla\phi,\nabla h\rangle,$$
	we have
	\begin{eqnarray}\label{equation16}
	\int_M\phi^2\left(2k+2\mu+2|A|^2\right)h^2e^{-f}d\sigma&=&4 \int_M\phi h\langle\nabla\phi,\nabla h\rangle e^{-f}d\sigma+2\int_M\phi^2|\nabla h|^2 e^{-f}d\sigma\nonumber\\
	&\leq&2\int_M h^2e^{-f}d\sigma+4\int_M|\nabla h|^2e^{-f}d\sigma.
	\end{eqnarray}
	Finally, consider the intrinsic ball  $B_j=B_j(p)$ in $M$ of radius $j$  and center at a fixed point $p\in M$. By applying (\ref{equation16}) with $\phi=\phi_j$, where $\phi_j$ is one on $B_j$ and cuts off linearly to zero from $\partial B_j$ to $\partial B_{j+1},$ letting $j\rightarrow\infty,$ and using the monotone convergence theorem we obtain 
	\begin{equation*}
	\int_M|A|^2h^2e^{-f}d\sigma\leq\int_M\left((1-k-\mu)h^2+2|\nabla h|^2\right)e^{-f}d\sigma.
	\end{equation*}
	Now, we will prove the inequality (\ref{equation3}). In fact,   $\log h$ is well-defined and 
	\begin{eqnarray}\label{equation17}
	\Delta_f\log h&=&\frac{1}{h}\Delta_fh-|\nabla\log h|^2=\frac{1}{h}L_fh-|A|^2-k-|\nabla\log h|^2\nonumber\\
	&=&-\mu-|A|^2-k-|\nabla\log h|^2.
	\end{eqnarray}
	Let $\psi$ be a function with compact support. Thus, it follows from Lemma \ref{Green} and expression (\ref{equation17}) that 
		\begin{eqnarray}\label{equation20}
		\int_M\langle\nabla\psi^2,\nabla\log h\rangle e^{-f}d\sigma&=&-\int_M\psi^2(\Delta_f\log h)e^{-f}d\sigma\nonumber\\
		&=&\int_M\psi^2\left(\mu+|A|^2+k+|\nabla\log h|^2\right)e^{-f}d\sigma.
		\end{eqnarray}
	Combining (\ref{equation20}) with the following inequality
	\begin{equation*}
	\langle\nabla\psi^2,\nabla\log h\rangle\leq2|\nabla\psi|^2+\frac{1}{2}\psi^2|\nabla\log h|^2,
	\end{equation*}
	gives us this
	\begin{equation}\label{equation4}
	\int_M\psi^2\left(2|A|^2+|\nabla\log h|^2\right)e^{-f}d\sigma\leq\int_M\left(4|\nabla\psi|^2-2\mu\psi^2-2k\psi^2\right)e^{-f}d\sigma.
	\end{equation}
	Let $\psi_j$ be one on $B_j$ and cut off linearly to zero from $\partial B_j$ to $\partial B_{j+1}$. Since $\phi\in W^{1,2}$, applying (\ref{equation4}) with $\psi=\psi_j\phi,$ and letting $j\rightarrow\infty,$ using the monotone convergence theorem, we get
	\begin{equation*}
	\int_M\phi^2\left(2|A|^2+|\nabla\log h|^2\right)e^{-f}d\sigma\leq\int_M\left[4|\nabla\phi|^2-2(\mu+k)\phi^2\right]e^{-f}d\sigma.
	\end{equation*}
\end{proof}

\begin{remark}
	Lemma \ref{l2} was obtain by Colding and Minicozzi II (see \cite{ColdingMinicozzi2012}, Lemma 9.15) to a complete noncompact hypersurface $\Sigma\subset\R^{n+1}$ without boundary that satisfies $$H=\frac{\langle x,\eta\rangle }{2},$$
	where $x$ is position vector.
\end{remark}  
Through use the Lemmas \ref{lemma2} and \ref{l2}, and ideas as the proof of Lemma 9.25 from \cite{ColdingMinicozzi2012} we have

\begin{lemma}\label{proposition4}
	Let $\mu_1(M)$ be the bottom $L^2_f$ spectrum of $L_f$. If $\mu_1(M)\neq-\infty$, then there exists a positive $C^2$ function $u$ on $M$ with $L_f u=-\mu_1(M)u.$ Moreover, if $w\in W^{1,2}_f$ and $L_f w=-\mu_1(M) w,$ then $w=Cu$ for some $C\in\R.$
\end{lemma}

\begin{lemma}\label{lemma6}
	Let $M$ be a complete oriented constant weighed mean curvature hypersurface in gradient Ricci soliton $\overline M_f$. If $\mu_1(M)\neq-\infty$ and $\Vol_f(M)<\infty$, then
	\begin{equation}
	\int_M|A|^2e^{-f}d\sigma<\infty.
	\end{equation}
\end{lemma}

\begin{proof}
	 Since $\mu_1(M)\neq-\infty$, there is a $C^2$ positive function $h$ on $M$ satisfying $L_fh=-\mu_1(M)h$ by Lemma \ref{proposition4}. Let $\phi_j$ be the cut off functions such that $|\nabla\phi_j|\leq1$ and $\phi_j\leq1$. So $\phi_j\in W^{1,2}(e^{-f}d\sigma)$ because $\Vol_f(M)<\infty$. By Lemma \ref{l2}, equality (\ref{equation3}), we have
	\begin{eqnarray*}
	2\int_M\phi_j^2|A|^2e^{-f}d\sigma&\leq&\int_M\phi_j^2\left(2|A|^2+|\nabla\log h|^2\right)e^{-f}d\sigma\\
	&\leq&\int_M\left(4|\nabla\phi_j|^2-2(\mu_1(M)+k)\phi_j^2\right)e^{-f}d\sigma\\
	&\leq&\left(4+2|\mu_1(M)+k|\right)\int_Me^{-f}d\sigma<\infty.
	\end{eqnarray*}
	By letting $j\rightarrow\infty,$ we obtain the conclusion this lemma by monotone convergence theorem. 
\end{proof}

The compact manifolds which admit a parallel vector field with respect to some metric was characterized by Welsh in \cite{Welsh1986a}. Namely, they are the compact fibre bundles over tori with finite structural group. The dimension of the torus can be assumed to be the number of linearly independent parallel vector fields.

The noncompact case has also been solved by Welsh in \cite{Welsh1986}. If fact, 

\begin{proposition}[\cite{Welsh1986},  Proposition 2.1]\label{proposition0}
	If a Riemannian manifold $M$ admits a complete parallel vector field, then either $M$ is diffeomorphic to the product of a Euclidian space with some other manifold or there is a circle action on $M$ whose orbits are not real homologous to zero. 
\end{proposition}

\begin{proof}[Proof of Theorem \ref{itheorem2}.]
	Now suppose that $\dim \mathcal P_{\overline M_f}=l>0$ and by hypothesis
	$$\Ind_f(M)=\dim \mathcal P_{\overline M_f}-\dim\{X\in P_{\overline M_f}: \ \langle X,\eta\rangle\equiv0\}.$$ We can have two cases: either $\Ind_f(M)=\dim \mathcal P_{\overline M_f}$ or $\Ind_f(M)\neq\dim \mathcal P_{\overline M_f}.$ 
	
	Initially, suppose that $\Ind_f(M)=\dim \mathcal P_{\overline M_f}$, that is $\dim\{X\in\mathcal P_{\overline M_f}: \ \langle X,\eta\rangle\equiv0\}=0$ and there exists $l$ independent linearly parallel unit vector fields $X_1,X_2, \ldots,X_l$ such that $\langle X_j,\eta\rangle\not\equiv0$ for all $j=1,2,\ldots, l$. Put $$u_j=\langle X_j,\eta\rangle \ \ \ \text{for all} \ j=1,2,\ldots,l.$$ By Proposition \ref{proposition1}, $$L_fu_j=ku_j,$$
	where $\overline{\Ric}_f=kg$. Note that 
	\begin{equation}\label{eq4}
	\int_Mu_j^{2}e^{-f}d\sigma\leq\int_Me^{-f}d\sigma<\infty
	\end{equation}
	because
	$$u_j=\langle X_j,\eta\rangle\leq|X_j||\eta|=1.$$
	 Hence, $u_1,u_2,\ldots,u_l$ are $L^2(e^{-f}d\sigma)$ eigenfunctions with negative eigenvalue $-k$. Since, by hypothesis, $$\Ind_f(M)=l>0,$$ then 
	 the bottom $\mu_1(M)$ of the spectrum of $L_f$ satisfies  $\mu_1(M)=-k.$ 
	 
	 Now, consider a local orthonormal frame $\{e_1,e_2,\ldots,e_n\}$ on $M$. Observe that
	 \begin{eqnarray}\label{equation18}
	 |\nabla u_j|^2=\sum_{i=1}^n|\nabla_{e_i}u_j|^2\geq \sum_{i=1}^n\left(\sum_{k=1}^n a_{ik}^2\right)\left(\sum_{k=1}^n\langle e_k,X_j\rangle^2\right)\leq |A|^2.
	 \end{eqnarray}
    Hence, it follows from Lemma \ref{lemma6} and inequality (\ref{equation18}) that
	$$\int_M|\nabla u_j|^{2}e^{-f}d\sigma\leq\int_M|A|^2e^{-f}d\sigma<\infty$$
	and by using (\ref{eq4}), we get $u_j\in W^{1,2}(e^{-f}d\sigma)$. By Lemma \ref{proposition4}, $u_j>0$ on $M$ without lost of generality. Therefore, it follows from Lemma \ref{lema5} and the integrability of $|A|^2$, that $$\int_M|A|^2u_je^{-f}d\sigma=0.$$ Thus, $|A|\equiv0$ on $M$. 
	
In the case which $\Ind_f(M)\neq\dim \mathcal P_{\overline M_f}$, we have there exists a $u_0\equiv0,$ where $u_0=\langle X_0,\eta\rangle$ and $X_0\in T\overline M_f$ is parallel field with respect to Riemannian connection of $(\overline{M}^{n+1},\overline g)$, in particular, $X_0\in TM$ is parallel with respect to Riemannian connection of $(M,g)$. By Proposition \ref{proposition0}, either $M$ is diffeomorphic to the product of a Euclidian space with some other manifold, or there is a circle action on $M$ whose orbits are not real homologous to zero.   

\end{proof}

\vspace{0.5cm}

\noindent
{\bf Hilário Alencar}\\
Instituto de Matem\'atica\\
Universidade Federal de Alagoas\\
Macei\'o, AL, 57072-900, Brazil\\
\email{hilario@mat.ufal.br}
\vspace{0.5cm}

\noindent
{\bf Adina Rocha}\\	
Instituto de Matem\'atica\\
Universidade Federal de Alagoas\\
Macei\'o, AL, 57072-900, Brazil\\
\email{adina@pos.mat.ufal.br}
\vspace{0.5cm}

\end{document}